\documentclass[11pt]{amsart}
\usepackage{amsmath,amssymb,amscd,amsthm,verbatim,alltt,amsfonts,array}
\usepackage{mathrsfs}
\usepackage[english]{babel}
\usepackage{latexsym}
\usepackage{amssymb}
\usepackage{euscript}
\usepackage{graphicx}

\newtheorem{theorem}{Theorem}[section]
\theoremstyle{plain}

\newtheorem{lemma}[theorem]{Lemma}

\numberwithin{equation}{section}

\title[]{Sectionally positive curvature tensors admit a metric tensor under which they are Einstein}
\author{Dan Gregorian Fodor}
\address{Faculty of Mathematics, Alexandru Ioan Cuza University, Ia\c si, Romania }

\begin{document}
\maketitle

\begin{abstract}
Let $n \geq 3$ and  $R_{abcd}$ be a $(4,0)$ sectionally positive curvature-type tensor (a tensor possessing all the local symmetries of the $(4,0)$ curvature tensor). Then there exists a metric tensor $g_{ab}$ such that $R_{abcd}\; g^{bd} = g_{ac} \lambda$ for some $\lambda$. Furthermore, $g_{ab}$ is unique up to a constant factor.
\end{abstract}

\section{Introduction}
Transformations of the curvature tensor that simplify its structure while maintaining desired properties are an interesting area of study (eg. \cite{AB}, pg 49-51, canonical form of the curvature tensor in dimension $4$). Starting from the fact that a $(4,0)$ curvature-type tensor taken as an independent object will maintain its sectional positivity even under a different metric, we aim to find in those metrics under which the given $(4,0)$ tensor has the simplest form. To this end, our main theorem characterises those metrics under which the tensor is Einstein (its Ricci tensor is proportional to the metric), including some conditions for their existence and uniqueness.

\section{Main result}
\begin{theorem}
Let $R_{abcd}$ be a $(4,0)$ curvature-type tensor, and $\mathcal{S}$ be the space of $(0,2)$ metric tensors $g^{ab}$ of determinant 1. Define $R : \mathcal{S} \rightarrow \mathbb{R}$  by $R(g^{ab}) = R_{abcd} (g^{ac}g^{bd})$ as the scalar curvature under the metric $g^{ab}$. The set of determinant $1$ metrics under which $R_{abcd}$  is Einstein (ie. $R_{abcd}g^{ac} = \lambda g_{bd}$) is the set of critical points of $R$. If  $R_{abcd}$ is strictly sectionally positive, such a point exists and is unique, corresponding to the minimum of $R$.

\end{theorem}

\begin{theorem}
Let $R_{abcd}$ be a $(4,0)$ curvature-type tensor which is sectionally positive under a metric $h$. Then it will remain sectionally positive under another metric $g$. Sectional positivity does not depend on the choice of metric.
\end{theorem}

\begin{proof}
A tensor $R_{abcd}$ can be defined as being sectionally positive if \mbox{$R_{abcd}v^{a}q^{b}v^{c}q^{d} \geq 0$}, for all non-collinear non-zero pairs of contravariant vectors $(v,q)$. It can be defined as strictly sectionally positive  if $R_{abcd}v^{a}k^{b}v^{c}k^{d} > 0$ . These definitions do not depend on a metric. The choice of metric does however alter the values of sectional curvature, which is defined by $K(v,q) = \frac{R_{abcd}v^{a}q^{b}v^{c}q^{d}}{(g_{ab}v^{a}v^{b})(g_{cd}q^{c}q^{d})  - (g_{ab}v^{a}q^{b})^{2}}$. They key to linking the above definitions is noting that the denominator $(g_{ab}v^{a}v^{b})(g_{cd}q^{c}q^{d})  - (g_{ab}v^{a}q^{b})^{2}$ is always positive for non-zero non-collinear $(v,q)$, regardless of choice of metric $g_{ab}$.
\end{proof}

\phantom{s\\}
\noindent We will be searching for solutions in the $\frac{n(n+1)}{2}$-dimensional space  $M$ of contravariant metric tensors $g^{ab}$. We endow this space with the natural metric given by
$$\langle d^{a_1 b_1}, d^{a_2 b_2} \rangle_{g} = g_{a_{1}a_{2}} g_{b_{1}b_{2}}d^{a_1 b_1}d^{a_2 b_2}$$
This gives us the geodesic equation ${\ddot g}^{ab} = {\dot g}^{ax}{\dot g}^{yb}g_{xy}$ with geodesics of the form: \begin{equation}
m^{ab}(t) = g^{ax} e^{t(g_{xy}d^{yb})}\label{eq}
\end{equation}

We know that if a metric $g_{ab}$ makes the curvature-type tensor $R_{abcd}$ Einstein, the same can be said for $\lambda (g_{ab})$ where $\lambda > 0$. To eliminate this degree of freedom, we will be working on the space $\mathcal{S}$ of metrics with fixed determinant, $\det(g^{ab})=1$. This is a complete totally geodesic metric space under the metric inherited \mbox{from $M$}.

\begin{theorem} \label{t1}
 Define $R : \mathcal{S} \rightarrow \mathbb{R}$  by $R(g^{ab}) = R_{abcd} (g^{ac}g^{bd})$, ie. the scalar curvature obtained from our tensor in the given metric. Then the critical points of $R$ are exactly those metrics of determinant $1$ under which $R_{abcd}$ is Einstein.
\end{theorem}

\begin{proof}
We compute the gradient of $R$ as its variation for $g^{ab} + hd^{ab}$, $h \sim 0$.\\ We have $R(g^{ab} + hd^{ab}) = R_{abcd} (g^{ac} + hd^{ac})(g^{bd} + hd^{bd}) =  R(g^{ab}) + 2hR_{ab} d^{ab}$ . \\This gives $\nabla_{d^{ab}}R = 2 R_{ab}d^{ab} = 2(\frac{1}{n}R d^{ab}g_{ab} + T_{ab}d^{ab})$ where $T_{ab}$ is the traceless part of the Ricci tensor. However, we are working over the submanifold $\mathcal{S}$ of $M$, where we require $\det(g^{ab}) = \det(g^{ab} + hd^{ab}) = 1$ . This translates to $g_{ab}d^{ab} = 0$, giving us  $\nabla_{d^{ab}}R = 2T_{ab}d^{ab}$. Critical points occur when $T_{ab} =0$, giving $R_{ab} = \lambda g_{ab}$.
\end{proof}

\noindent We now compute some bounds for $R(g^{ab})$ over the space of \mbox{metrics of determinant $1$}.
\begin{lemma}
Let $R_{abcd}$ be a sectionally positive curvature operator (with metric $\delta_{ab}$), and $R_{s}$ be its sectional curvature along a plane $s$. Then $2 R_{s} \leq R$, ie. $\frac{R}{2R_{s}}\geq 1$.
\end{lemma}
\begin{proof}
We may select an orthonormal basis such that $R_{s} = R_{1212}$. $$R = R_{abcd}(\delta^{ac}\delta^{bd}) = \sum\limits_{k=1}^{n}\sum\limits_{p=1}^{n} R_{kpkp}= 2(\sum\limits_{k<p} R_{kpkp}) = 2R_{1212} + 2\sum\limits_{\tiny{\substack{k<p\\(k,p)\neq(1,2)}}} R_{kpkp}$$\\
We note that the second term is a sum of sectional curvatures, and therefore greater than or equal to $0$. Therefore $R \geq 2R_{1212}$.
\end{proof}

\begin{lemma} \label{l0}
Let $g^{ab}$ be a metric such that its two largest eigenvalues are $\lambda_{1}$, $\lambda_{2}$. Then, for sectionally positive $R_{abcd}$, we have $R(g^{ab}) \geq 2\lambda_{1}\lambda_{2}R_{s}$ where $R_{s}$ is the smallest sectional curvature of $R_{abcd}$ under the default metric $\delta_{ab}$.

\end{lemma}

\begin{proof}
The smallest eigenvalues of $g_{ab}$ are $\frac{1}{\lambda_{1}}$, $\frac{1}{\lambda_{2}}$. Let $v_{1}$, $v_{2}$ be the corresponding eigenvectors. The sectional curvature over the plane spanned by $\{v_1,v_2\}$ for the metric $g_{ab}$ is given by $R_{\lambda} =\frac{R_{abcd}v_{1}^{a}v_{2}^{b}v_{1}^{c}v_{2}^{d}}{(g_{ab}v_{1}^{a}v_{1})(g_{cd}v_{2}^{c}v_{2}^{d})  - (g_{ab}v^{a}_{1}v_{2}^{b})^{2}} = (R_{1212}) \lambda_{1}\lambda_{2} \geq  \lambda_{1}\lambda_{2}R_{s}$. However, by the previous lemma, $R(g^{ab}) \geq 2R_{\lambda}$. This gives $R(g^{ab})\geq 2\lambda_{1}\lambda_{2}R_{s}$, thus completing the proof.
\end{proof}

\begin{lemma} \label{l1}
Let $g^{ab}$ be a metric tensor of determinant $1$ and greatest eigenvalue $\lambda_{1}$. Then the product of its two greatest eigenvalues is greater than or equal to $\lambda_{1}^{\frac{n}{n-1}}$.
\end{lemma}

\begin{proof}
As the determinant is $1$, the product of the remaining eigenvalues is $\frac{1}{\lambda_{1}}$. Our aim is to minimize the largest of them. This happens when all are equal to $\frac{1}{\lambda_{1}}^{\frac{1}{n-1}}$, giving a product of $\lambda_{1}^{\frac{n}{n-1}}$.

\end{proof}

\begin{lemma} \label{l2}
Let $g^{ab}$ be a metric tensor of determinant $1$ and smallest \mbox{eigenvalue $\lambda_{2}$}. Then the product of its two largest eigenvalues is greater than or equal too $\lambda_{2}^{-\frac{2}{n-1}}$.
\end{lemma}

\begin{proof}
The product of the remaining $n-1$ eigenvalues is $\frac{1}{\lambda_{2}}$. The product of the two largest of them is minimized when they are all equal. This gives $\lambda_{2}^{-\frac{2}{n-1}}$.
\end{proof}

\begin{lemma}
Let $R_{abcd}$ be a strictly sectionally positive curvature operator and $R_{s}$ be its smallest sectional curvature under the metric $\delta^{ab}$. Let $N$ be the set of metrics of determinant $1$ whose smallest eigenvalue is greater than or equal to $(\frac{R(\delta^{ab})}{2R_{s}})^{-\frac{n-1}{2}}$ and largest eigenvalue is smaller than or equal to $(\frac{R(\delta^{ab})}{2R_{s}})^{\frac{n-1}{n}}$. Then, for any metric $g^{ab}\in \mathcal{S}/N$, we have $R(g^{ab}) > R(\delta^{ab})$.

\end{lemma}

\begin{proof}
If $g^{ab}\in \mathcal{S}/N$ then $g^{ab}$ satisfies one of two conditions. \\1) Its largest eigenvalue is greater than $(\frac{R(\delta^{ab})}{2R_{s}})^{\frac{n-1}{n}}$. Then by Lemma \ref{l1} the product of its two largest eigenvalues is greater than $\frac{R(\delta^{ab})}{2R_{s}}$.\\
2)Its smallest eigenvalue is smaller than \ref{l2}. Then by Lemma \ref{l2} the product of its two largest eigenvalues is greater than $\frac{R(\delta^{ab})}{2R_{s}}$.

\noindent Applying Lemma \ref{l0} for all of the above cases shows that \mbox{$R(g^{ab})> 2\frac{R(\delta^{ab})}{2R_{s}} R_{s} = R(\delta^{ab})$}.
\end{proof}

\noindent We are now in a position to prove the existence of a critical point for $R$, namely its minimum point. We know $R(g^{ab}) > R(\delta^{ab})$ for any $g^{ab}$ outside $N$. The set of metric tensors $g^{ab}$ for which $R(g^{ab}) \leq R(\delta^{ab})$ is a non-zero subset of the compact set $N$ (we know $\delta^{ab} \in N$). Therefore $R$ must achieve at least one metric $m^{ab} \in \mathcal{S}$  which minimizes it. That metric will be a critical point, so, by Theorem \ref{t1}, $R_{abcd}$ will be Einstein under $m^{ab}$. Uniqueness is equivalent to showing the minimum is the only critical point of $R$. To prove its uniqueness we require an additional lemma.

\begin{lemma} For strictly sectionally positive $R_{abcd}$, the hessian of $R:\mathcal{S} \rightarrow \mathbb{R}$ at any point is strictly positive definite.

\end{lemma}

\begin{proof}
We determine the hessian by computing the second-order variation in $R$ along a geodesic. From the geodesic equation $m^{ab}(t) = g^{ax} e^{t(g_{xy}d^{yb})}$ (Equation \ref{eq}), we get $m^{ab}(h) = g^{ax} e^{h(g_{xy}d^{yb})} \sim g^{ab} + hd^{ab} + \frac{h^{2}}{2}d^{ax}g_{xy}d^{yb}$, where $h\sim 0$ . Expanding on $R_{abcd}m^{ac}m^{bd}$
we obtain $$R_{abcd}m^{ac}m^{bd} = R + 2hR_{ab}d^{ab} + h^{2}(R_{ab}d^{ax}g_{xy}d^{yb} + R_{abcd}d^{ac}d^{bd})$$ We see that the hessian when applied to $d^{ab}$ gives $2(R_{ab}g_{cd} + R_{abcd})d^{ac}d^{bd}$. We mow prove this quantity is always greater than $0$ for $d^{ab} \neq 0$. We may choose coordinates under which $g_{ab} = \delta_{ab}$ and $d^{ab} = \sum\limits_{k=1}^{n} \lambda_{k} ({v_{k}}^{a}{v_{k}}^{b})$ where $\{v_{1},v_{2}\cdots v_{n}\}$ form an orthonormal frame. Under these coordinates we may write $$2(R_{ab}g_{cd} + R_{abcd})d^{ac}d^{bd} = \sum\limits_{k<p}2(\lambda_{k}+\lambda_{p})^{2}R_{abcd}({v_{k}}^{a}{v_{p}}^{b}{v_{k}}^{c}{v_{p}}^{d})= \sum\limits_{k<p}2(\lambda_{k}+\lambda_{p})^{2}R_{kpkp}$$
We have $(\lambda_{k}+\lambda_{p})^{2}R_{kpkp} \geq 0$. The result is greater than or equal to $0$, and it is equal to $0$ if and only if $\forall k,p, k<p$ we have $(\lambda_{k}+\lambda_{p})^{2}=0$. But, for $n \geq  3$ this implies $\forall {k}, \lambda_{k} = 0$, theretofore $d^{ab}=0$. We have: $$d^{ab}\neq 0 \implies 2(R_{ab}g_{cd} + R_{abcd})d^{ac}d^{bd} > 0$$ This completes the proof of the strict positivity of the hessian.
\end{proof}

We may now prove the minimum critical point is unique. Consider $m_{1}, m_{2}$ two distinct critical points and let $p(t)$ be a parameterised geodesic uniting them, with $p(0) = m_{1}$ and $p(1) = m_{2}$. We have $\dot p(0) = 0$, and $\ddot p(t) > 0$, due to the strict positivity of the hessian. By integrating, we obtain $\dot p(1) > 0$. Therefore $m_{2} = p(1)$ can not be a critical point. From this we obtain our main theorem.

%footnotesize {\noindent Faculty of Mathematics, Alexandru Ioan Cuza University, Ia\c si, Romania }

\end{document}